\documentclass[12pt,reqno]{amsart}
\usepackage{amssymb}
\usepackage[all]{xy}
\usepackage{ulem}

\newtheorem{theorem}{Theorem}[section]
\newtheorem{proposition}[theorem]{Proposition}
\newtheorem{lemma}[theorem]{Lemma}

\theoremstyle{definition}

\newtheorem{remark}[theorem]{Remark}
\newtheorem{example}[theorem]{Example}
\newtheorem{question}[theorem]{Question}
\numberwithin{equation}{section}

\begin{document}

\baselineskip=15pt

\title[Positivity of vector bundles on homogeneous varieties]{Positivity of vector 
bundles on homogeneous varieties}

\author[I. Biswas]{Indranil Biswas}

\address{School of Mathematics, Tata Institute of Fundamental
Research, Homi Bhabha Road, Mumbai 400005, India}

\email{indranil@math.tifr.res.in}

\author[K. Hanumanthu]{Krishna Hanumanthu}

\address{Chennai Mathematical Institute, H1 SIPCOT IT Park, Siruseri, Kelambakkam 603103, 
India}

\email{krishna@cmi.ac.in}

\author[D. S. Nagaraj]{D. S. Nagaraj}

\address{Indian Institute of Science Education and Research, Tirupati,
Karakambadi Road, Mangalam (P.O.) Tirupati -517507,
Andhra Pradesh, India}

\email{dsn@iisertirupati.ac.in}



\date{August 3, 2020}

\begin{abstract}
We study the following question: Given a vector bundle on a projective variety $X$ such 
that the restriction of $E$ to every closed curve $C \,\subset\, X$ is ample, under what 
conditions $E$ is ample? We first consider the case of an abelian variety $X$. If $E$ is 
a line bundle on $X$, then we answer the question in the affirmative. When $E$ is of 
higher rank, we show that the answer is affirmative under some conditions on $E$. We then 
study the case of $X \,=\, G/P$, where $G$ is a reductive complex affine algebraic group, and $P$ is a 
parabolic subgroup of $G$. In this case, we show that the answer to our question is 
affirmative if $E$ is $T$--equivariant, where $T\, \subset\, P$ is a fixed maximal torus.
Finally, we compute the Seshadri constant for such vector bundles defined on $G/P$.
\end{abstract}

\maketitle

\tableofcontents

\section{Introduction}\label{se1}

Let $X$ be a projective variety defined over an algebraically closed field, and let $L$ be a line 
bundle on $X$. The Nakai--Moishezon criterion says that $L$ is ample if and only if 
$L^{\text{dim}(Y)}\cdot Y \,>\, 0$ for every positive-dimensional closed subvariety $Y$ of $X$. In 
general, it is not sufficient to check this condition only for the closed curves in $X$. Mumford 
gave an example of a non-ample line bundle on a surface which intersects every closed curve 
positively; see \cite[Example 10.6]{Har} or \cite[Example 1.5.2]{La1}.

However, in some cases it turns out that to check ampleness of $L$ it suffices to verify 
the condition $L\cdot C \,>\, 0$ for all closed curves $C \,\subset\, X$. Line bundles
satisfying this condition are called \textit{strictly nef}; see
\cite{Se}. Strictly nef divisors have been studied by many
authors. Strictly nef divisors have interesting connections to many questions; for
more details, see \cite{CCP,Se}.

Mumford's example gives a strictly nef, but not ample,
divisor on a ruled surface. One can still ask the following question:

Under what situations is a strictly nef divisor ample?

Examples where the answer is known to be positive are provided by 
abelian varieties and toric varieties. 

In this note, we ask the following:

\textit{Given a vector bundle $E$ on a projective variety $X$ such 
that the restriction of $E$ to every closed curve $C \,\subset\, X$ is ample, under what 
conditions is $E$ ample?}

In \cite{HMP}, this question is studied for toric varieties; in fact, in \cite{HMP} it is 
proved that an equivariant vector bundle $E$ on a toric variety $X$ is ample if the 
restriction of $E$ to the invariant rational curves on $X$ is ample. We recall that there 
are only finitely many invariant rational curves on $X$. For a flag variety $X$ over a 
projective curve defined over $\overline{\mathbb F}_p$, a line bundle on $X$ is ample if its 
restriction to each closed curve is ample \cite{BMP}. When $X$ is a wonderful compactification, 
this question is studied in \cite{BKN}.

Here we address the above question for abelian varieties and homogeneous varieties $G/P$,
where $G$ is a reductive affine algebraic group over $\mathbb C$, and $P$ is a parabolic 
subgroup of $G$.

In Section \ref{se2}, we consider abelian
varieties. The case of line bundles on abelian
varieties is known, but we start with by giving a proof in this case,
for completeness (Proposition \ref{thm1}). We then consider vector
bundles on abelian varieties and answer the question in the
affirmative under some conditions (see Proposition \ref{prop1}). Our result
shows, in particular, that a homogeneous vector bundle on
an abelian variety has the above mentioned property. 

In Section \ref{se3}, we consider homogeneous varieties $G/P$, 
where $G$ and $P$ are as above. Let $T$ be a maximal torus of $G$ contained
in $P$. We show that a $T$--equivariant vector bundle has positive answer to our
question (see Theorem \ref{G/P}). Finally, we calculate Seshadri constants
for $T$--equivariant bundles on $G/P$ at $T$--fixed points (see Theorem \ref{sc}).

\section{Vector bundles on an abelian variety}\label{se2}

Let $k$ be an algebraically closed field. We first consider the case
of line bundles on abelian varieties. In this case it is known that
our question has a positive answer; see \cite[Proposition 1.4]{Se} for
example\footnote{We thank Patrick Brosnan for pointing out this
reference to us.}. We still include a proof below for the sake of completeness.

\begin{proposition}\label{thm1}
Let $A$ be an abelian variety defined
over $k$. Let $L$ be a line bundle over $A$ satisfying the following condition: for every pair
$(C,\, f)$, where $C$ is an irreducible smooth projective curve defined over $k$, and
$f\, :\, C\, \longrightarrow\, A$ is a non-constant morphism, the inequality
\begin{equation}\label{e3}
{\rm degree}(f^*L)\, >\, 0
\end{equation}
holds.
Then $L$ is ample.
\end{proposition}

\begin{proof}
Take a line bundle $\mathcal L$ on $A$. Let
$$
\alpha\, :\, A\times A \, \longrightarrow\, A\, , \ \ (x,\, y)\, \longmapsto\, x+y
$$
be the addition map. Consider the family of line bundles
$$
(\alpha^*{\mathcal L})\otimes p^*_1 {\mathcal L}^*\, \longrightarrow\, A\times A\,
\stackrel{p_2}{\longrightarrow}\, A\, ,
$$
where $p_1$ and $p_2$ are the projections of $A\times A$ to the first and second factor
respectively. Let
$$
\varphi_{\mathcal L}\, :\, A\, \longrightarrow\, A^\vee\,=\, {\rm Pic}^0(A)
$$
be the classifying morphism for this family. This $\varphi_{\mathcal L}$ is a group homomorphism.
Let
$$
K({\mathcal L})\, \subset\, A
$$
be the (unique) maximal connected subgroup of the reduced kernel $\text{ker}
(\varphi_{\mathcal L})_{\rm red}$.

If ${\mathcal L}\, \in\, A^\vee \,=\, {\rm Pic}^0(A)$, then $\varphi_{\mathcal L}$ is the constant morphism
$x\, \longmapsto\, 0$ \cite[p.~120]{MFK} (see after Definition 6.2), \cite[p.~11, Lemma 2.1.6]{GN}. Using this it follows that if ${\mathcal L}'$ is numerically equivalent
to $\mathcal L$, then $\varphi_{\mathcal L}\,=\, \varphi_{{\mathcal L}'}$, which in turn implies
that
\begin{equation}\label{e0}
K({\mathcal L})\,=\, K({\mathcal L}')\, .
\end{equation}
It is known that $\mathcal L$ is ample if the following two conditions hold:
\begin{enumerate}
\item the line bundle $\mathcal L$ is effective, and

\item $K({\mathcal L})\,=\, 0$.
\end{enumerate}
(See \cite[p.~288, \S~1]{Mum1}, \cite[p.~13, Theorem 2.2.2]{GN}.)

We will use the following lemma:

\begin{lemma}\label{lem1}
The line bundle $L$ in Proposition \ref{thm1} is ample if $K(L)\,=\, 0$.
\end{lemma}

\begin{proof}[{Proof of Lemma \ref{lem1}}]
Since $L$ is nef, it follows that $L$ is numerically equivalent to a $\mathbb Q$--effective
$\mathbb Q$--Cartier divisor on $A$ (see \cite[p.~811, Proposition 3.1]{Mo}). So
$L^{\otimes n}$ is numerically equivalent to an effective divisor $D$ on $A$, for some
positive integer $n$. Note that
\begin{equation}\label{e1}
\varphi_{L^n}\,=\, n\cdot \varphi_L\, .
\end{equation}

Assume that $K(L)\,=\, 0$. Consequently, from \eqref{e1} and \eqref{e0} it follows that
$K(D)\,=\, 0$. Since $D$ is also effective, using the above mentioned criterion for
ampleness it follows that $D$ is ample. This implies that $L$ is ample.
\end{proof}

Continuing with the proof of Proposition \ref{thm1}, in view of Lemma \ref{lem1}
it suffices to show that $\dim K(L) \, =\, 0$. Assume that $$\dim K(L) \, \geq\, 1\, .$$
The restriction of $L$ to the sub-abelian variety $K(L)\,\subset\, A$ will be denoted
by $L_0$. For any closed point $x\, \in\, A$, define
\begin{equation}\label{e4}
\alpha_x\,:\, A\, \longrightarrow\, A\, , \ \ y\, \longmapsto\, x+y\, .
\end{equation}
For any closed point $x\, \in\, K(L)$, let $\widehat{\alpha}_x\, :\, K(L)\,
\longrightarrow\, K(L)$ be the restriction of $\alpha_x$ in \eqref{e4} to $K(L)$.

For any $x\, \in\, K(L)$, we have $\alpha^*_x L\,=\, L$; hence we have
$$\widehat{\alpha}^*_x L_0\,=\, (\alpha^*_x L)\vert_{K(L)}\,=\, L\vert_{K(L)}
\,=\, L_0\, .$$ This implies that the line bundle $L_0$ on $K(L)$
is numerically trivial \cite[p.~74, Definition]{Mum2} and \cite[p.~86]{Mum2}. Consequently, for
any pair $(C,\, f)$, where $C$ is an irreducible smooth projective curve defined over $k$, and
$f\, :\, C\, \longrightarrow\, K(L)\, \subset\, A$ is a non-constant morphism, we have
$$
\text{degree}(f^*L)\, =\, 0\, .
$$
Since this contradicts \eqref{e3}, we conclude that
$\dim K(L) \, =\, 0$. As observed above, this implies that $L$ is ample.
\end{proof}

\subsection{Ample vector bundles on $A$}

As before $A$ is an abelian variety.
Let $E$ be a vector bundle of rank $r$ over $A$ satisfying the following two conditions:
\begin{enumerate}
\item The line bundle $\det E\, :=\, \bigwedge^r E$ has the property that
for every pair $(C,\, f)$, where $C$ is an irreducible smooth projective curve defined over $k$, and
$f\, :\, C\, \longrightarrow\, A$ is a non-constant morphism, the inequality
$\text{degree}(f^*\det E)\, >\, 0$ holds.

\item for every closed point $x\, \in\, A$, there is a line bundle $L(x)$ on $A$ such that
\begin{equation}\label{e5}
\alpha^*_x E\,=\, E\otimes L(x)\, ,
\end{equation}
where $\alpha_x$ is the morphism in \eqref{e4}.
\end{enumerate}

\begin{proposition}\label{prop1}
The above vector bundle $E$ on $A$ is ample.
\end{proposition}

\begin{proof}
Since $E$ satisfies the condition in \eqref{e5}, a theorem of Mukai says that
there is an isogeny
$$
f\, :\, B\, \longrightarrow\, A
$$
such that the vector bundle $f^*E$ admits a filtration of subbundles
\begin{equation}\label{e6}
0\,=\, E_0\, \subset\,E_1\, \subset\, \, \cdots \, \subset\, E_{r-1} \, \subset\, E_r\,=\, E
\end{equation}
for which $\text{rank}(E_i)\,=\, i$, and the line bundle $E_i/E_{i-1}$ is numerically
equivalent to $E_1$ for every $1\,\leq\, i\, \leq\, r$ \cite[p.~260, Theorem 5.8]{Muk} (see also
\cite[p.~2]{MN}).

Consequently, the line bundle $\det E\, =\, \bigwedge^r E\,=\, \bigotimes_{i=1}^r
(E_i/E_{i-1})$ is numerically equivalent to the line bundle $E^{\otimes r}_1$. From
Proposition \ref{thm1} we know that $\det E$ is ample. This implies that $E^{\otimes r}_1$
is ample. Hence $E_1$ is ample. So $E_i/E_{i-1}$ is ample for every $1\,\leq\, i\, \leq\, r$.
Consequently, from \eqref{e6} it follows that $E$ is ample \cite[p.~13, Proposition 6.1.13]{La}.
\end{proof}

Let $X$ be a projective variety defined over an algebraically closed field $k$. 
A divisor $D$ on $X$ is said to be \textit{big} if there is an ample 
divisor $H$ on $X$ such that the difference $mD-H$ is linearly equivalent to an 
effective divisor for some positive integer $m$. A $\mathbb{Q}$--divisor $D$ 
is \textit{pseudo-effective} if $D+B$ is big for any big $\mathbb{Q}$--divisor $B$. 
Similarly one can define the notion of pseudo-effective $\mathbb{R}$--divisors. 
In the N\'eron--Severi space $N^1(X)_{\mathbb{R}}$, the pseudo-effective $\mathbb{R}$--divisors 
form a cone which is the closure of the cone of effective $\mathbb{R}$--divisors. 

If $\dim X\,=\,2$, and the pseudo-effective cone of $X$ is equal to
the cone of effective divisors, then 
a line bundle $L$ on $X$ is ample if and only if $L\cdot C \,>\, 0$ for every closed curve $C$ on $X$. 
But, in general, the pseudo-effective cone of a projective variety is not equal to the 
effective cone; see the example of Mumford described in \cite[Example 10.6]{Har} or 
\cite[Example 1.5.2]{La1}.

If $k$ is an algebraic closure of a finite field, Moriwaki showed that every 
pseudo-effective divisor (over $\mathbb{Q}$ or $\mathbb{R}$) is effective when $X$ is a 
projective bundle over a projective curve or when $X$ is an abelian variety (see
\cite[p.~802, Theorem 0.4]{Mo} and \cite[p.~802, Proposition 0.5]{Mo}). As our next example shows, 
this statement is false for abelian varieties over $\mathbb{C}$.

\begin{example}
Let $X$ be an elliptic curve defined over $\mathbb{C}$. Let $x\,\in\, X$ be a point of infinite order. 
Let $D\,:=\,0-x$, where $0$ is the identity element of $X$. Then $D$ is a divisor of degree 0 and 
it is pseudo-effective. However, no multiple of $D$ is effective, since $x$ has infinite order.
\end{example}

In view of Proposition \ref{prop1}, it is natural to ask the following:

\begin{question}\label{qab}
Let $A$ be an abelian variety over an algebraically closed field. Let $E$ be a vector
bundle on $A$ such that the restriction $E|_C$ is ample for every closed curve $C$ on $A$.
Is this $E$ ample?
\end{question}

\section{Equivariant vector bundles on $G/P$}\label{se3}

\begin{theorem}\label{G/P}
Let $G$ be a reductive affine algebraic group defined over $\mathbb C$, and let $P\,\subset\,
G$ be a parabolic subgroup. Fix a maximal torus $T\, \subset\, G$ such that $T\, \subset\,P$.
Let $E$ be a $T$--equivariant vector bundle on $G/P$. Then $E$ is nef
(respectively, ample) if and only if the restriction $E|_C$ of $E$
to every $T$--invariant closed curve $C$ on $G/P$ is nef (respectively, ample). 
\end{theorem}

\begin{proof}
Let $Y(T)$ be the group of all $1$--parameter subgroups of $T$. Note that
$Y(T)$ is a finitely generated free abelian group whose rank is equal to the
dimension of $T$. Let
\begin{equation}\label{b}
\{\lambda_1,\,\cdots,\,\lambda_n\}
\end{equation}
be a basis of the $\mathbb Z$--module $Y(T)$. 

If $E$ is nef (respectively, ample), then clearly $E|_C$ is nef (respectively, ample) for every closed
curve $C$ on $G/P$.

To prove the converse, first assume that $E$ is a $T$--equivariant vector
bundle on $G/P$ such that its restriction $E|_C$ to every
$T$--invariant curve $C\,\subset\, G/P$ is nef.

Let $$\pi\,:\, \mathbb{P}(E) \,\longrightarrow\, G/P$$ be the projective bundle
over $G/P$ parametrizing the hyperplanes in the fibers of $E$. The tautological
relative ample line bundle over
$\mathbb{P}(E)$ will be denoted by $\mathcal{O}_{\mathbb{P}(E)}(1)$.
To prove that $E$ is nef, we need to
show that $\mathcal{O}_{\mathbb{P}(E)}(1)|_D$ is nef for every closed curve
$D\,\subset\, \mathbb{P}(E)$. Note that if $\pi(D)$ is a point, then 
$\mathcal{O}_{\mathbb{P}(E)}(1)|_D$ is ample, because 
$\mathcal{O}_{\mathbb{P}(E)}(1)$ is relatively ample. 

Therefore, we can assume that $\pi(D)$ is a closed curve in $G/P$. Let $\widetilde{D_1}$ be
the flat limit of the curves $\lambda_1(t)D$ (see \eqref{b}) as $t \,\to\, 0$. In other
words, $\widetilde{D_1}$ is a $1$--cycle which corresponds to the limit of the points
$\lambda_1(t)[D]$ (as $t \,\to\, 0$) in the Hilbert scheme of curves in $\mathbb{P}(E)$. 
Note that since $E$ is $T$--equivariant, the $1$--parameter subgroup
$\lambda_1$ acts on the Hilbert scheme of curves in
$\mathbb{P}(E)$. It follows that the $1$--cycle $\widetilde{D_1}$ 
and $D_1: \,= \, \pi(\widetilde{D_1})$ are $\lambda_1$--invariant. Now let
$\widetilde{D_2}$ be the flat limit of $\lambda_2(t)\widetilde{D_1}$
as $t \,\to\, 0$. Since $\lambda_1$ and $\lambda_2$ commute, we see that 
$\widetilde{D_2}$ and $\pi(\widetilde{D_2})$ are invariant under both
$\lambda_1$ and $\lambda_2$. Continuing this way, we obtain a $1$--cycle
$\widetilde{D_n}$ on $\mathbb{P}(E)$ such that both 
$\widetilde{D_n}$ and $\pi(\widetilde{D_n})$ are invariant under
$\lambda_1,\,\cdots, \,\lambda_n$. Consequently, both 
$\widetilde{D_n}$ and $\pi(\widetilde{D_n})$ are $T$--invariant. 

Now, by the assumption on $E$, we have
\begin{equation}\label{in}
{\rm degree}(\mathcal{O}_{\mathbb{P}(E)}(1)|_{\widetilde{D_n}}) \,\ge\, 0\, .
\end{equation}
Since the curve $D$ is linearly equivalent to $\widetilde{D_n}$, from \eqref{in} it follows immediately
that ${\rm degree}(\mathcal{O}_{\mathbb{P}(E)}(1)|_D)\,\ge\, 0$. This proves that $E$ is nef if $E|_C$ is 
nef for every closed $T$--invariant curve $C \,\subset\, G/P$.

Next we shall prove that $E$ is ample if $E|_C$ is ample for every $T$--invariant
closed curve $C \,\subset\, G/P$.

Note that every line bundle on $G/P$ is
$T$--equivariant. Hence if $F$ is a $T$--equivariant vector bundle on $G/P$, then so is
$F \otimes L$ for any line bundle $L$ on $G/P$. 

We claim that the set of $T$--invariant closed curves in $G/P$ is
finite.

To prove the above claim, let $B$ be a Borel subgroup of $G$ containing $T$
and contained in $P$. Then $B$ acts on $G$ via left--translations.
Let $W\,:=\, N_G(T)/T$ be the corresponding Weyl group, and let $W_P \,\subset\, W$
be the subgroup consisting of elements that preserve $P$. Note that $G$ is the disjoint union of
double cosets $BwP$, where $w$ runs through $W/W_P$. This gives the
Bruhat decomposition $$G/P \,=\, \cup_{w \in W/W_P} BwP$$
(see \cite[\S~29.2]{Hu}, \cite[\S~29.3]{Hu}). It is clear that
the $T$--fixed points in $G/P$ are precisely $wP$ for every $w \,\in\,
W/W_P$. Further, in each of the $B$--orbits $BwP$ in $G/P$, there are
only finitely many $T$--invariant curves, namely the
images of the subgroups generated by the root vectors in $B$. Since $W$ is
finite, we conclude that $G/P$ has only finitely many $T$--fixed points
and $T$--invariant curves. This proves the claim.

The ampleness of $E$ now follows by a standard
argument, which we include for the convenience of the reader. Fix an
ample line bundle $L$ on $G/P$. 
Let $\text{Sym}^n(E)$ denote
the $n$--th symmetric power of $E$. Then for $n$ sufficiently large,
we have that $\text{Sym}^n(E)\otimes L^{-1}|_C$ is nef for every $T$--invariant curve $C$. By
the first part of the theorem, the vector bundle $\text{Sym}^n(E)\otimes L^{-1}$ is nef. Since
$L$ is ample, this implies that $\text{Sym}^n(E)\,=\, (\text{Sym}^n(E)\otimes L^{-1})\otimes L$
is ample, and consequently $E$ itself is ample (see \cite[Proposition 6.2.11]{La} and 
\cite[p.~67, Proposition 2.4]{Ha}). 
\end{proof}

\begin{proposition}\label{blowup}
Let $G$, $P$ and $T$ be as in Theorem \ref{G/P}.
Let $x \,\in\, G/P$ be a $T$--fixed point, and let $\varpi\,:\, \widetilde{X}\,\longrightarrow\,
G/P$ denote the blow-up of $G/P$ at $x$. Then the following three statements hold:
\begin{enumerate}
\item The action of $T$ lifts to $\widetilde{X}$. 

\item Let $F$ be a $T$--equivariant vector bundle on $\widetilde{X}$. Then $F$ is nef if and 
only if the restriction $F|_{\widetilde{C}}$ of $F$ to every $T$--invariant closed curve 
$\widetilde{C}\,\subset\, X$ is nef.

\item Let $W_x$ denote the the exceptional divisor of the blow-up $\varpi$. Let $E$ be a 
$T$--equivariant vector bundle on $G/P$. Then $(\varpi^{\star}E)\otimes 
\mathcal{O}_{\widetilde{X}}(W_x)^m$ is a $T$--equivariant vector bundle on $\widetilde{X}$ for 
every integer $m$.
\end{enumerate}
\end{proposition}

\begin{proof}
Since $x$ is a $T$--fixed point, the group $T$ acts on the tangent space
$T_x(G/P) = \mathfrak{g}/\mathfrak{p}$, where $\mathfrak{g}$ is the
Lie algebra of $G$ and $\mathfrak{p}$ is the Lie algebra of $P$. Then
the $T$--action on $T_x(G/P)$ decomposes it into a direct sum of
one-dimensional $T$--invariant subspaces.

Note that the exceptional divisor of the blow-up $\varpi\,:\, \widetilde{X} \,\longrightarrow\, G/P$ is
isomorphic to $\mathbb{P}(T_x(G/P))$. So $T$ acts on the exceptional divisor
of the blow-up via the linear action of $T$ on $T_x(G/P)$. Since $\varpi$
is an isomorphism outside the exceptional divisor, we conclude that
the action of $T$ lifts to all of $\widetilde{X}$. This proves (1). 

Since the action of $T$ lifts to $\widetilde{X}$, the proof of (2) goes through
exactly as the proof of the analogous statement in Theorem
\ref{G/P}. Note that in the proof of Theorem \ref{G/P} we only used
the $T$--action on $G/P$. 

Now we prove (3). 
Since $T$ acts on the exceptional divisor $W_x$ of $\varpi$, 
we conclude that $\mathcal{O}_{\widetilde{X}}(W_x)^m$ is a $T$--equivariant line
bundle for every integer $m$. Hence if $E$ is a $T$--equivariant
vector bundle on $G/P$ then $(\varpi^{\star}E) \otimes \mathcal{O}_{\widetilde{X}}(W_x)^m$ is a 
$T$--equivariant vector bundle on $\widetilde{X}$ for every integer $m$. 
\end{proof}

Let $E$ be a vector bundle on a projective variety $X$. Take any point $x \,\in\,
X$. The Seshadri constant of $E$ at $x$ was defined in
\cite{Hac}. This definition is recalled below. 

Let $\varpi\,:\, \widetilde{X} \,\longrightarrow\, X$ denote the blow-up of $X$ at $x$.
Consider the following diagram, where $p$ and $q$ are projective bundles:
\begin{equation}\label{fig1}
\xymatrix{
\mathbb{P}(\varpi^{\star}E) \ar[r]^{\widetilde{\varpi}}\ar[d]_{q} & \mathbb{P}(E)\ar[d]^{p}\\
\widetilde{X} \ar[r]^{\varpi} & X
 }
\end{equation}

Let
\begin{equation}\label{xi}
\xi \,=\,
\mathcal{O}_{\mathbb{P}(\varpi^{\star}E)}(1)
\end{equation}
be the tautological bundle on $\mathbb{P}(\varpi^{\star}E)$. Let 
$Y_x \,=\, p^{-1}(x)$ and $Z_x \,=\,\widetilde{\varpi}^{-1}(Y_x)$.

The Seshadri constant of $E$ at $x$ is defined as follows: 
$$\varepsilon(E,x) \,:=\, \text{sup}\{\lambda \,\in \, \mathbb{Q}_{> 0} ~\mid ~ \xi - \lambda
Z_x\ \text{~is nef}\}.$$

Here $\mathbb{Q}_{>0}$ denotes the set of positive rational numbers. 
For more details on Seshadri constants of vector bundles, see
\cite{Hac}. 

Now we work with the notation in Theorem \ref{G/P}.
Let $E$ be a $T$--equivariant vector bundle on $G/P$. It is known that
each $T$--invariant closed curve $C \,\subset\, G/P$ is smooth rational. Indeed, let
$C$ be a $T$--invariant closed curve. Since there are only finitely many
$T$--fixed points in $G/P$, there must exist a point $x \,\in\, C$ which is
not fixed by $T$. Now consider the morphism $T \,\longrightarrow\, C$ which sends $t
\in T$ to $t\cdot x \,\in\, C$. This is a non-constant morphism from a
torus to $C$, and hence $C$ must be rational. In the special case when 
$G \,=\,{\rm GL}_n({\mathbb C})$, and $P$ is the Borel subgroup of upper triangular matrices
in $G$, a different proof for this fact can be
found in \cite[Page 44, 1.3.4, Example 2]{Bri}. 

Let $C \,\subset\, G/P$ be a
$T$--invariant closed curve. From a theorem of Grothendieck we know
that the restriction of $E$ to $C$ has the form
$$E|_C\, =\, \mathcal{O}_C(a_1) \oplus \ldots \oplus \mathcal{O}_C(a_n)$$ for some
integers $a_1,\,\cdots,\,a_n$ \cite[p.~122, Th\'eor\`eme 1.1]{Gr}.

\begin{theorem}\label{sc}
Let $G$ be a complex reductive group, and let $P$ be a parabolic
subgroup of $G$ containing a maximal torus $T$. Let $E$ be a
$T$--equivariant nef vector bundle on $X \,=\, G/P$ of rank $n$, and let $x \,\in\, X$ be a $T$--fixed
point. Then 
$$\varepsilon(E,x) \,=\, \text{min}\{a_i(C)\}_{C,i}\, ,$$ where the minimum is taken over
all $T$--invariant curves $C \,\subset\, G/P$ passing through $x$ and integers
$\{a_1(C),\,\cdots,\, a_n(C)\}$ such that $E|_C \,= \,\mathcal{O}_C(a_1(C)) \oplus \cdots \oplus
\mathcal{O}_C(a_n(C))$. 
\end{theorem}

\begin{proof}
Let $W_x$ denote the exceptional divisor of the blow-up $$\varpi\,:\, \widetilde{X} \,\longrightarrow\,
X \,=\, G/P$$ at the point $x \,\in\, G/P$. By definition, the Seshadri constant of $E$ at
$x$ is given by the following: 
$$\varepsilon(E,x) \,=\, \text{sup}\{\lambda \in \mathbb{Q}_{> 0} ~\mid ~ \xi - \lambda q^{\star}(W_x)\
\text{~is nef}\}\, ,$$
where $q$ is the map in \eqref{fig1}, $\xi$ is the line bundle in \eqref{xi} and $W_x$
is the exceptional divisor as in the proof of Proposition \ref{blowup}(3).

We claim that $\xi - \lambda q^{\star}(W_x)$ is nef 
if $\mathcal{O}_{\mathbb{P}(\varpi^{\star}E|_{\widetilde{C}})}(1)-\lambda
q_1^{\star}(W_x|_{\widetilde{C}})$ 
is nef for every $T$--invariant closed curve $\widetilde{C} \,\subset\,
\widetilde{X}$. See the following diagram:
$$
\xymatrix{
 \mathbb{P}(\varpi^{\star}E|_{\widetilde{C}}) \ar[r]\ar[d]_{q_1} & \mathbb{P}(\varpi^{\star}E)
\ar[r]^{\widetilde{\varpi}}\ar[d]_{q} & \mathbb{P}(E)\ar[d]^{p}\\
 \widetilde{C} \ar@{^{(}->}[r] & \widetilde{X} \ar[r]^{\varpi} &
 X = G/P
 }
$$

The above claim essentially follows
from the proof of Theorem \ref{G/P} and Proposition \ref{blowup}(2). 
Indeed, to prove that $\xi - \lambda q^{\star}(W_x)$ is nef, we need
to show that $(\xi - \lambda q^{\star}(W_x)) \cdot D \,\ge\, 0$ for every
closed curve $D \,\subset\, \mathbb{P}(\varpi^{\star}E)$. But, as
argued in the proof of Theorem \ref{G/P}, there exists a $T$--invariant
curve $\widetilde{D} \,\subset\, \mathbb{P}(\varpi^{\star}E)$ which is linearly
equivalent to $D$ and such that 
$\widetilde{C} \,:=\, q(\widetilde{D}) \,\subset\, \widetilde{X}$ is a $T$--invariant curve. 
But we have 
$$(\xi - \lambda q^{\star}(W_x)) \cdot \widetilde{D}\,=\,
\text{degree}(E|_{\widetilde{C}}) - \lambda (W_x \cdot \widetilde{C})
\,\ge\, 0\, .$$ The last inequality follows from the hypothesis that 
$\mathcal{O}_{\mathbb{P}(\varpi^{\star}E|_{\widetilde{C}})}(1)-\lambda
q_1^{\star}(W_x|_{\widetilde{C}})$ is nef for every invariant curve 
$\widetilde{C} \,\subset\, \widetilde{X}$.
This proves the claim.

Now let $\widetilde{C} \,\subset\, \widetilde{X}$ be any $T$--invariant
curve. We know that $\widetilde{C}$ is isomorphic to the projective
line $\mathbb{P}^1$. 
We will show below that $\mathcal{O}_{\mathbb{P}(\varpi^{\star}E|_{\widetilde{C}})}(1)-\lambda
q_1^{\star}(W_x|_{\widetilde{C}})$ is nef. 

First suppose that $\widetilde{C}$ is contained in the exceptional
divisor $W_x$ of the blow-up $\varpi\,:\, \widetilde{X} \,\longrightarrow\, X$. 
Note that $W_X$ is isomorphic to a projective space, and 
$$\mathcal{O}_{W_x}(W_x) \,=\, \mathcal{O}_{W_x}(-1)\, .$$ Hence we have
$-\lambda(W_x|_{\widetilde{C}}) \,=\,\mathcal{O}_{\widetilde{C}}(\lambda)$. 
Since $E$ is nef on $X$ by hypothesis, it follows that 
$\mathcal{O}_{\mathbb{P}(\varpi^{\star}E|_{\widetilde{C}})}(1)-\lambda
q_1^{\star}(W_x|_{\widetilde{C}})$ is nef for every $\lambda \,>\, 0$. 

Now suppose that $\widetilde{C}$ is not contained in $W_x$, and let $C
\,=\, \varpi(\widetilde{C})$. Then $C \,\subset\, X$ is a $T$--invariant curve. If
$x\,\notin\, C$, then $W_X|_{\widetilde{C}} \,=\,
\mathcal{O}_{\widetilde{C}}$, and 
$\mathcal{O}_{\mathbb{P}(\varpi^{\star}E|_{\widetilde{C}})}(1)$ is nef
because $\varpi^{\star}E|_{\widetilde{C}}$ is so.

So assume that $x \,\in\, C$. Then $W_x \cdot \widetilde{C} \,=\, 1$, since $C$ is a
smooth curve. Let $a_1(C),\,\cdots,\,a_n(C)$ be non-negative integers such that
$E|_C \,=\, \mathcal{O}_C(a_1(C)) \oplus \ldots \oplus \mathcal{O}_C(a_n(C))$. Then 
$\mathcal{O}_{\mathbb{P}(\varpi^{\star}E|_{\widetilde{C}})}(1)-\lambda
q_1^{\star}(W_x|_{\widetilde{C}})$ is nef if and only if 
$\mathcal{O}_{\widetilde{C}}(a_1(C)-\lambda) \oplus \cdots \oplus \mathcal{O}_{\widetilde{C}}(a_n(C)-\lambda)$ is nef. 
Now. $\mathcal{O}_{\widetilde{C}}(a_1(C)-\lambda) \oplus \cdots \oplus
\mathcal{O}_{\widetilde{C}}(a_n(C)-\lambda)$ is nef
if and only if $\lambda \,\le\,
\text{min}\{a_1(C),\,\cdots,\,a_n(C)\}$. Running over all $T$--invariant curves in
$\widetilde{X}$ the theorem is proved. 
\end{proof}

\begin{remark}
A similar computation of Seshadri constants was carried out for equivariant
vector bundles on toric varieties in \cite[Proposition 3.2]{HMP} which
motivated our result. 
\end{remark}

The following is well-known, but we give this example to show how our results give a
simpler argument. 

\begin{example}
Let $0 < d < n$ be integers and let $X \,=\, {\rm Gr}(d,n)$ be the Grassmannian
of $d$--dimensional subspaces of $\mathbb{C}^n$. Then one has the
universal exact sequence
$$0\,\longrightarrow\, S\,\longrightarrow\, X \times \mathbb{C}^n
\,\longrightarrow\, Q \,\longrightarrow\, 0$$ of vector
bundles on $X$, where the fiber of $S$ over a point $x \,\in\, X$
is the $d$--dimensional subspace $S_x$ of $\mathbb{C}^n$
corresponding to $x \,\in\, X$ while the fiber of $Q$ over $x$ is the
quotient vector space $\mathbb{C}^n/S_x$. Further, it is easy to prove that
all three vector bundles in the above exact sequence are ${\rm GL}_n({\mathbb C})$--equivariant, and
hence they are $T$--equivariant, where $T \,\subset\, {\rm GL}_n({\mathbb C})$ is the subgroup
of diagonal matrices.

Now let $C \,\subset\, X$ be a
$T$--invariant curve. Then $Q|_C \,=\, \mathcal{O}_C(1) \oplus
\mathcal{O}_C^{\oplus(n-d-1)}$. Note that $C \,\cong\, {\mathbb C}\mathbb{P}^1$. 
Hence $Q|_C$ is always nef and it is ample if and only if $n-d\,=\,1$. So, by Theorem
\ref{G/P}, the vector bundle $Q$ itself is always nef and it is ample if and only if $n-d\,=\,1$.
Moreover, if $Q$ is ample, then $\varepsilon(Q,x) \,=\, 1$ for every $T$--fixed point $x \,\in\, X$ by
Theorem \ref{sc}.

The tangent bundle $\mathcal{T}_X$ is isomorphic to
$Hom(S,Q) \,=\, S^{\vee} \otimes Q$, where $S^{\vee}$ denotes the dual of
$S$. Arguing as above, we conclude that $\mathcal{T}_X$ is ample if
and only if either $S$ or $Q$ is a line bundle which is the case
precisely when $d\,=\,1$ or $n-d\,=\,1$. Of course, this happens if and only
if $X$ is the projective space ${\mathbb C}\mathbb{P}^n$. 

Finally note that the determinant bundle $\det (Q)$ is $T$--equivariant,
because $Q$ is so. If $C$ is any $T$--invariant closed curve in $X$, then
$\det (Q)|_C \,=\, \mathcal{O}_C(1)$. Hence the vector bundles $Q \otimes
\text{det}(Q)$ and $S^{\vee} \otimes \text{det}(Q)$ are both ample. 
\end{example}

\section*{Acknowledgements}

The first author is supported by a J. C. Bose Fellowship, and school of mathematics, TIFR, is
supported by 12-R$\&$D-TFR-5.01-0500.
The second author is partially supported by DST SERB MATRICS grant MTR/2017/000243
and also a grant from Infosys Foundation. The authors thank the
National Institute of Science Education and Research (NISER),
Bhubaneswar for hospitality while a part of this work was carried out. 


\end{document}